\documentclass[12pt]{amsart}

\usepackage{amssymb}
\usepackage{verbatim}
\usepackage[toc,page]{appendix}
\usepackage{mathrsfs}
\usepackage{yfonts}

\newtheorem{thm}{Theorem}[section]

\newtheorem{lem}[thm]{Lemma}

\newtheorem*{rem*}{Remark}

\theoremstyle{definition}

\numberwithin{equation}{section}

\newcommand{\fr}{\kappa}

\newcommand{\Z}{\mathbf{Z}}

\newcommand{\Mod}[1]{\ (\textup{mod}\ #1)}

\DeclareMathOperator{\res}{res}

\begin{document}

\title[]{Sums of Kloosterman sums in the prime geodesic theorem}

\author{Olga  Balkanova}
\address{Department of Mathematical Sciences, Chalmers University of Technology and University of Gothenburg, Chalmers tv\"{a}rgata 3, 412 96 Gothenburg, Sweden}
\email{olgabalkanova@gmail.com}

\author{Dmitry  Frolenkov}
\address{
National Research University Higher School of Economics, Moscow, Russia and
Steklov Mathematical Institute of Russian Academy of Sciences,  8 Gubkina st., Moscow, 119991, Russia}
\email{frolenkov@mi-ras.ru}
\thanks{Research of Dmitry Frolenkov is supported by the Foundation for the advancement of theoretical physics and Mathematics "BASIS" and by the Program of the Presidium of the Russian Academy of Sciences ¹01 'Fundamental Mathematics and its Applications' under grant PRAS-18-01.}

\begin{abstract}
We develop a new method for studying sums of Kloosterman sums related to the spectral exponential sum. As a corollary, we  obtain a new proof of the estimate of Soundararajan and Young for the error term in the prime geodesic theorem.
\end{abstract}


\maketitle

\tableofcontents

\section{Introduction}

The goal of this paper is to prove the following estimate for the sums of Kloosterman sums.
\begin{thm}\label{mainthm}
Let $\theta$ be a subconvexity exponent for Dirichlet $L$-functions of real primitive characters and let $S(n,n,q)$ denote the Kloosterman sum.
For $X,T\gg 1$ the following estimate holds
\begin{multline}\label{sum of Kloosterman22}
\frac{1}{N}\sum_{n}h(n)\sum_{q=1}^{\infty}\frac{S(n,n;q)}{q}\varphi\left(\frac{4\pi n}{q}\right)\ll\\ \max(X^{1/4+\theta/2}T^{3/2},X^{\theta/2}T^2)\log^{2}(XT)+
\frac{X^{1/4+\theta}T^{3/2}}{N^{1/2}}\left(1+\frac{T}{X^{1/2}}\right),
\end{multline}
where
\begin{equation}\label{phi def0}
\varphi(x)=\frac{\sinh^2\beta}{2\pi}x^2\exp(ix\cosh\beta)
\end{equation}
with
\begin{equation}\label{beta def}
\beta=\frac{1}{2}\log X+\frac{i}{2T}
\end{equation} and $h(x)$ is a smooth function supported in $[N,2N]$ for some $N>1$ such that
\begin{equation}\label{h conditions}
|h^{(j)}(x)|\ll N^{-j},\,\hbox{for}\,j=0,1,2,\ldots\, \int_{-\infty}^{\infty}h(x)dx=N.
\end{equation}
\end{thm}
\begin{rem*}
Throughout the paper we use the standard notation $A\ll B$ meaning that there  exists a constant $c>1$ such that $A\le cB.$
On the other hand, by writing $E\gg F$ we mean that there  exists a constant $c>1$ such that $E\ge cF.$
\end{rem*}
\begin{rem*}
Applying the Weil bound for Kloosterman sums, we obtain (for $N\gg TX^{-1/2}$) the estimate
\begin{equation}\label{sum of Kloosterman22 Weil}
\frac{1}{N}\sum_{n}h(n)\sum_{q=1}^{\infty}\frac{S(n,n;q)}{q}\varphi\left(\frac{4\pi n}{q}\right)\ll N^{1/2}X^{1/4}T^{3/2}\log(NX).
\end{equation}
Therefore, \eqref{sum of Kloosterman22} improves \eqref{sum of Kloosterman22 Weil} if
$$N\gg X^{\theta}\left(1+TX^{-1/2}\right).$$
\end{rem*}
The sum \eqref{sum of Kloosterman22} is particularly interesting because it is ultimately related to  the prime geodesic theorem, as we now explain.

The prime geodesic theorem gives an asymptotic formula as $X \rightarrow \infty$ for the number  $\pi_{\Gamma}(X)$  of primitive hyperbolic classes $\{P\}$ in $PSL_{2}(\Z)$ with norm $NP$  less than or equal to $X$.
In direct analogy with prime numbers, it is convenient to study the weighted counting function
$$\Psi_{\Gamma}(X)=\sum_{NP\leq X}\Lambda(P),$$ where the sum is over all  hyperbolic classes and $\Lambda(P)=\log{NP_0}$ if $\{P\}$ is a power of the primitive hyperbolic class $\{P_0\}$.

Iwaniec \cite[Lemma 1]{IwPG} proved for $1\le T\le X^{1/2}\log^{-2}X$ that
\begin{equation}\label{PrimeGeodesic to spec.sum}
\Psi_{\Gamma}(X)=X+2X^{1/2}\Re\left(\sum_{0<t_j\le T}\frac{X^{it_j}}{1/2+it_j}\right)+
O\left(\frac{X}{T}\log^2X\right),
\end{equation}
where $\kappa_j=1/4+t_{j}^{2}$ are the eigenvalues of the hyperbolic Laplacian for $PSL_2(\Z)$.
Exact formula  \eqref{PrimeGeodesic to spec.sum} provides a connection between $\Psi_{\Gamma}(X)$ and the spectral exponential sum
\begin{equation}\label{spec.exp.sum}
S(T,X)=\sum_{0<t_j\le T}X^{it_j}.
\end{equation}
It follows from Weyl's law that $S(T,X)\ll T^{2}$, and correspondingly $$\Psi_{\Gamma}(X)=X+O(X^{3/4}\log X).$$ The first non-trivial estimate
\begin{equation}\label{Iwaniec bound}
S(T,X)\ll TX^{11/48+\epsilon}
\end{equation}
was obtained by Iwaniec in \cite{IwPG} for $1\le T\le X^{1/2}\log^{-2}X$. Using \eqref {Iwaniec bound} and taking $T=X^{13/48}$ in \eqref{PrimeGeodesic to spec.sum},  we have
\begin{equation}\label{Iwaniec bound2}
\Psi_{\Gamma}(X)=X+O(X^{13/48+\epsilon}).
\end{equation}
In order to prove \eqref{Iwaniec bound},  Iwaniec showed that the problem can be reduced to the investigation of the smoothed sum
\begin{equation}\label{spec.exp.sum2}
\sum_{j}X^{it_j}\exp(-t_j/T).
\end{equation}
Consequently, using properties of the Rankin-Selberg $L$-function and introducing the additional parameter $N$, Iwaniec obtained the following decomposition
\begin{equation}\label{PrG =A+B+O}
\sum_{j}X^{it_j}\exp(-t_j/T)=\mathcal{A}+\mathcal{B}
+
O\left(T\log^2T+\frac{N^{1/2}\log^2N}{X^{1/2}}\right)
\end{equation}
with
\begin{equation}\label{A def}
\mathcal{A}=\frac{\pi^2}{12N}\sum_{n}h(n)\sum_{q=1}^{\infty}\frac{S(n,n;q)}{q}
\phi\left(\frac{4\pi n}{q}\right),
\end{equation}
\begin{equation}\label{B def}
\mathcal{B}=-\frac{\zeta(2)}{2N}
\frac{1}{2\pi i}\int_{(1/2)}\tilde{h}(s)\sum_{j}\hat{\phi}(t_j)\frac{L(u_{j}\otimes u_{j},s)}{\cosh(\pi t_j)}ds,
\end{equation}
where $h(x)$ is as in Theorem \ref{mainthm}, $\tilde{h}(s)$ is the Mellin transform of $h(x)$,
\begin{equation}\label{eq:phiLuoS}
\phi(x)=\frac{\sinh\beta}{\pi}x\exp(ix\cosh\beta), \quad
\beta=\frac{1}{2}\log X+\frac{i}{2T}
\end{equation}
and
\begin{equation}\label{phihat expression0}
\hat{\phi}(t)=\frac{\sinh(\pi t+2i\beta t)}{\sinh(\pi t)}=X^{it}\exp(-t/T)+O(\exp(-\pi t)).
\end{equation}

We remark that the original choice of $\phi(x)$ made by Iwaniec was different. However, the function \eqref{eq:phiLuoS} produces a smaller error term in the approximation \eqref{phihat expression0}, as has been shown by Deshouillers and Iwaniec in \cite[Lemmas 7 and 9]{DesIw}.

The next step is to optimise the choice of the parameter $N$ by proving the sharpest possible estimates on  $\mathcal{A}$ and $\mathcal{B}$.  To estimate the part $\mathcal{B}$, Iwaniec studied the first moment of Rankin-Selberg $L$-functions and proved an upper bound for this moment  that is slightly weaker than the mean Lindel\"{o}f estimate. Consequently,
\begin{equation}\label{Iwaniec Bbound}
\mathcal{B}\ll N^{-1/2}T^{5/2}\log^2T.
\end{equation}
Note that in order to break the $3/4$ barrier in the prime geodesic theorem, it is insufficient to combine \eqref{Iwaniec Bbound} with the trivial bound for the part $\mathcal{A}$, namely
\begin{equation}\label{LuoSarnak Abound}
\mathcal{A}\ll N^{1/2}T^{1/2}X^{1/4}\log T.
\end{equation}
For this reason,  Iwaniec analysed sums of Kloosterman sums in the part $\mathcal{A}$ using the Burgess bound for character sums, and taking $N$ sufficiently large ($N\gg T^{1+\epsilon}\gg X^{11/48+\epsilon}$), he finally proved \eqref{Iwaniec bound}.

The estimate  \eqref{Iwaniec bound} was further improved by Luo and Sarnak \cite{LuoSarnakPG} as follows
\begin{equation}\label{LuoSarnak bound}
S(T,X)\ll X^{1/8}T^{5/4}\log^2T, \quad T,X\gg1.
\end{equation}
More precisely, by proving the mean Lindel\"{o}f estimate for the first moment in the part $\mathcal{B}$, Luo and Sarnak showed that
\begin{equation}\label{LuoSarnak Bbound}
\mathcal{B}\ll N^{-1/2}T^{2}\log^2T.
\end{equation}
 Substituting \eqref{LuoSarnak Abound} and \eqref{LuoSarnak Bbound} to \eqref{PrG =A+B+O} yields that the optimal choice of $N$ is $N=T^{3/2}X^{-1/4}$.  This proves  \eqref{LuoSarnak bound} provided that $T>X^{1/6}.$  Finally, applying \eqref{LuoSarnak bound} to evaluate \eqref{PrimeGeodesic to spec.sum}, it turns out that the optimal choice of $T$ is $T=X^{3/10}$, and this gives
\begin{equation}\label{LuoSarnak bound2}
\Psi_{\Gamma}(X)=X+O(X^{7/10+\epsilon}).
\end{equation}
Consequently,   $N=X^{1/5}=T^{2/3}.$
Such a small value of $N$, on the one hand, makes it more difficult to show cancellations in sums of Kloosterman sums in the part $\mathcal{A}$, but on the other hand, it suffices to apply the trivial estimate \eqref{LuoSarnak Abound} based on the Weil bound in order to prove \eqref{LuoSarnak bound}.

The next improvement is due to Cai \cite{Cai}, who showed that for $ X^{1/10}\le T\le X^{1/3}$ the following estimate holds
\begin{equation}\label{Cai bound1}
S(T,X)\ll T^{2/5}X^{11/30+\epsilon}.
\end{equation}
The approach of \cite{Cai} combines the estimate \eqref{LuoSarnak Bbound}  for the part $\mathcal{B}$ and non-trivial analysis of the part $\mathcal{A}$ via the Burgess bound for character sums, which is possible because the parameter $N$ is sufficiently large, namely $N=T^{1+\epsilon}.$ Note that \eqref{Cai bound1} improves  \eqref{LuoSarnak bound} if $T>X^{29/102+\epsilon}.$ This means that Cai obtained a non-trivial estimate on the part $\mathcal{A}$ for $N>X^{29/102+\epsilon}.$  It follows from \eqref{PrimeGeodesic to spec.sum}, \eqref{Cai bound1} and \eqref{LuoSarnak bound} that
\begin{equation}\label{Cai bound2}
\Psi_{\Gamma}(X)=X+O(X^{71/102+\epsilon}).
\end{equation}

Finally, Soundararajan and Young \cite{SY} proved the prime geodesic theorem in the strongest known form
\begin{equation}\label{PrimeGeodesic bound}
\Psi_{\Gamma}(X)=X+O\left(X^{2/3+\theta/6}\log^3X\right),\quad  \theta=1/6+\epsilon.
\end{equation}
The proof is based on the estimate \eqref{LuoSarnak bound} and a formula relating $\Psi_{\Gamma}(X)$ with sums of generalized Dirichlet $L$-functions.

Another interesting question related to the prime geodesic theorem is the correct order of magnitude of the spectral exponential sum \eqref{spec.exp.sum}.
In \cite[Conj. 2.2]{PetRisLaak} Petridis and Risager conjectured that for $X,T\gg 1$
\begin{equation}\label{PetrisdisRisager conj}
S(T,X)\ll T(TX)^{\epsilon}.
\end{equation}
For a fixed $X$ and $T\to\infty$ the conjecture  was proved by Laaksonen in the appendix of \cite{PetRisLaak}.

The following estimates proved for $X,T \gg 1$ in \cite{BF}  confirm the conjecture of Petridis and Risager in some ranges

\begin{equation}\label{spec.exp.sum new bound}
S(T,X)\ll \max\left(
X^{1/4+\theta/2}T^{1/2},
X^{\theta/2}T
\right)\log^{3}T,
\end{equation}
\begin{equation}\label{spec.exp.sum new bound2}
S(T,X)\ll T\log^{2}T\quad\hbox{if}\quad T>\frac{X^{1/2+7\theta/6}}{\fr(X)},
\end{equation}
where
\begin{equation}\label{fr(X) def}
\fr(X)=\left\|X^{1/2}+X^{-1/2}\right\|
\end{equation}
and $\|x\|$ denotes the distance from $x$ to the nearest integer.
Furthermore, combining \eqref{spec.exp.sum new bound} with \eqref{LuoSarnak bound},  we obtain a new proof of \eqref{PrimeGeodesic bound}.

Estimates \eqref{spec.exp.sum new bound} and  \eqref{spec.exp.sum new bound2} follow from the nontrivial bound
for the part $\mathcal{B}$
\begin{multline}\label{symsquare estimate0}
\mathcal{B}\ll \frac{T\log^3T}{N^{1/2}}+\\
\frac{X^{\theta}}{X^{1/4}N^{1/2}}\left(
X^{1/2}\min\left(T,\frac{X^{1/2}}{\fr(X)}\right)^{1/2}+
\min\left(T,\frac{X^{1/2}}{\fr(X)}\right)^{3/2}\right),
\end{multline}
while the part $\mathcal{A}$ is estimated using Weil's bound as in \eqref{LuoSarnak Abound}. Consequently, the optimal choice of $N$ is $N=X^{\theta}.$

An obvious idea is to try to improve also the trivial estimate on $\mathcal{A}$. Nevertheless, it turns out that the Weil bound gives a stronger result in the required range than Iwaniec's method based on the Burgess bound for character sums.
Indeed, the estimate  \eqref{spec.exp.sum new bound} is better than \eqref{LuoSarnak bound} only if $T>X^{1/6+2\theta/3+\epsilon}$. In order to establish cancellations in character sums using the Burgess bound, it is required that $N\gg T$. Thus $N\gg T>X^{1/6+2\theta/3+\epsilon}$. However, the optimal choice of $N$ is $N=X^{\theta}$, which is much less than $X^{1/6+2\theta/3+\epsilon}.$

The aim of the present paper is to develop a new method for estimating the part $\mathcal{A}$, which is based on some ideas of Kuznetsov \cite{K}.  As a result, we obtain Theorem \ref{mainthm} and a new proof of \eqref{PrimeGeodesic bound}.

Due to some problems with convergence, it is required to replace the function $\phi(x)$ defined by \eqref{eq:phiLuoS} with $\varphi(x)$ (see \eqref{phi def0}), which decays faster at $x=0.$  As will be shown later, in this case
\begin{multline}\label{phihat expression1}
\hat{\varphi}(t)=t\frac{\cosh(\pi t+2i\beta t)}{\sinh(\pi t)}+
\frac{i\sinh(\pi t+2i\beta t)}{2\tanh\beta\sinh(\pi t)}=\\
tX^{it}\exp(-t/T)+\frac{i\cosh\beta}{2\sinh\beta}X^{it}\exp(-t/T)+O(\exp(-\pi t)).
\end{multline}
Consequently, we study the sum
$$\sum_{0<t_j\le T}t_jX^{it_j}$$
instead of \eqref{spec.exp.sum}.  The approach of Iwaniec \cite{IwPG} (see also \cite[Section 6]{LuoSarnakPG} for more details) yields
\begin{equation}\label{PrG1 =A+B+O}
\sum_{j}\hat{\varphi}(t_j)=\mathcal{A}_1+\mathcal{B}_1+
O\left(T^2\log^2T+\frac{N^{1/2}\log^2N}{X^{1/2}}\right)
\end{equation}
with
\begin{equation}\label{A1 def}
\mathcal{A}_1=\frac{\pi^2}{12N}\sum_{n}h(n)\sum_{q=1}^{\infty}\frac{S(n,n;q)}{q}
\varphi\left(\frac{4\pi n}{q}\right),
\end{equation}
\begin{equation}\label{B1 def}
\mathcal{B}_1=-\frac{\zeta(2)}{2N}
\frac{1}{2\pi i}\int_{(1/2)}\tilde{h}(s)\sum_{j}\hat{\varphi}(t_j)\frac{L(u_{j}\otimes u_{j},s)}{\cosh(\pi t_j)}ds.
\end{equation}
Note that $\mathcal{B}_1\ll T\mathcal{B}$, and therefore, all previous estimates on the part $\mathcal{B}$ are valid for $\mathcal{B}_1$ being multiplied by $T.$

In fact, the estimate \eqref{sum of Kloosterman22} is very strong as it allows us to prove the following theorem even by using  the original estimate \eqref{Iwaniec Bbound} of Iwaniec on the part $\mathcal{B}$, while the proofs given in \cite{SY} and \cite{BF} rely crucially on the mean Lindel\"{o}f estimate \eqref{LuoSarnak Bbound} by Luo and Sarnak.
\begin{thm}\label{thm:spec.exp.sum new bound2} For $X,T\gg 1$ the following holds
\begin{equation}\label{spec.exp.sum new bound3}
\sum_{t_j\le T}t_jX^{it_j}\ll \max\left(
X^{1/4+\theta/2}T^{3/2},
X^{\theta/2}T^2
\right)\log^{3}(XT).
\end{equation}
\end{thm}
Note that  \eqref{spec.exp.sum new bound}  (multiplied by $T$) and \eqref{spec.exp.sum new bound3} are of the same quality. Therefore, as a consequence, we obtain one more proof of the prime geodesic theorem in the strongest known form \eqref{PrimeGeodesic bound}. We remark that it is not possible to improve \eqref{PrimeGeodesic bound} further by combing \eqref{sum of Kloosterman22} with the strongest known estimate for the part $\mathcal{B}$.
The reason is that the largest contribution to the final result  comes from the first summand on the right-hand side of  \eqref{sum of Kloosterman22}, and this summand  does not depend on $N.$ This eliminates the possibility of improvement  by optimising the additional parameter $N$,  which was the main idea of Iwaniec's approach.


\section{Notation and preliminary results}\label{section:notation}

Introduce the following notation
\begin{equation*}
\mathop{{\sum}^*}_{n=0}^{\infty}a_n=\frac{a_0}{2}+\sum_{n=1}^{\infty}a_n.
\end{equation*}
For a function $f(x)$ let
\begin{equation*}
\tilde{f}(s)=\int_0^{\infty}f(x)x^{s-1}dx
\end{equation*}
be its Mellin transform.
Let $e(x)=exp(2\pi ix)$. The classical Kloosterman sum
\begin{equation*}
S(n,m;c)=\sum_{\substack{a\pmod{c}\\ (a,c)=1}}e\left( \frac{an+a^*m}{c}\right), \quad aa^*\equiv 1\pmod{c},
\end{equation*}
satisfies Weil's bound (see \cite[Theorem 4.5]{Iw})
\begin{equation}\label{Weilbound}
|S(m,n;c)|\leq \tau_0(c)\sqrt{(m,n,c)}\sqrt{c}.
\end{equation}
Consider the generalized Dirichlet $L$-function
\begin{equation}\label{Lbyk}
\mathscr{L}_{n}(s)=\frac{\zeta(2s)}{\zeta(s)}\sum_{q=1}^{\infty}\frac{\rho_q(n)}{q^{s}}=\sum_{q=1}^{\infty}\frac{\lambda_q(n)}{q^{s}},\quad \Re{s}>1,
\end{equation}
where
\begin{equation}
\rho_q(n):=\#\{x\Mod{2q}:x^2\equiv n\Mod{4q}\},
\end{equation}
\begin{equation}
\lambda_q(n):=\sum_{q_{1}^{2}q_2q_3=q}\mu(q_2)\rho_{q_3}(n).
\end{equation}

Zagier \cite[Proposition 3]{Z} showed that \eqref{Lbyk} can be meromorphically continued to the whole complex plane. Furthermore, it was proved in \cite[Proposition 3]{Z} that  $\mathscr{L}_n(s)$ is identically zero if $n \equiv 2,3 \Mod{4}.$ For $n=0$ we have
\begin{equation}
\mathscr{L}_{0}(s)=\zeta(2s-1).
\end{equation}
Otherwise, for $n=Dl^2$ with $D$ fundamental discriminant the following decomposition holds
\begin{equation}\label{ldecomp}
\mathscr{L}_{n}(s)=l^{1/2-s}T_{l}^{(D)}(s)L(s,\chi_D),
\end{equation}
where $L(s,\chi_D)$ is a Dirichlet $L$-function for primitive quadratic character $\chi_D$ and
\begin{equation}\label{eq:td}
T_{l}^{(D)}(s)=\sum_{l_1l_2=l}\chi_D(l_1)\frac{\mu(l_1)}{\sqrt{l_1}}\tau_{s-1/2}(l_2).
\end{equation}

It follows from \eqref{ldecomp} that if $n$ is non-zero and is not a full square, then $\mathscr{L}_{n}(s)$ is an entire function.
Another consequence of \eqref{ldecomp} is that for any $n$ and some constant $A>0$ one has
\begin{equation}\label{eq:subconvexity}
\mathscr{L}_n(1/2+it)\ll (1+|n|)^{\theta}(1+|t|)^{A},
\end{equation}
where $\theta$ is the best known result towards the Lindel\"{o}f hypothesis for Dirichlet $L$-functions of real primitive characters. Conrey and Iwaniec \cite{CI} showed that $\theta=1/6+\epsilon$ is admissible for any $\epsilon>0$ and Young \cite{Y} proved the hybrid bound with $\theta=A=1/6+\epsilon$.

For $V\ge1$ define the series
\begin{equation}\label{SV def}
S_V(n^2-4):=\sum_{q=1}^{\infty}\frac{\lambda_q(n^2-4)}{q}\exp(-q/V).
\end{equation}
It was shown in  \cite[Eqs. (1.6)-(1.8), p. 723]{B}, \cite[p. 116, line 2]{SY} that for $V>0$ and $n\neq2$ one has
\begin{equation}\label{approximate func.eq.}
\mathscr{L}_{n^2-4}(1)=
S_V(n^2-4)-\frac{1}{2\pi i}\int_{(-1/2)}\mathscr{L}_{n^2-4}(1+s)V^s\Gamma(s)ds.
\end{equation}
We will also use \cite[Lemma 2.3]{SY}, which states that for  $q=a^2b$ with $b$ square-free
\begin{equation}\label{sum of lambda}
\sum_{2<n\le z}\lambda_q(n^2-4)=z\frac{\mu(b)}{b}+O(q^{1/2+\epsilon})\text{ for any }z\ge 2.
\end{equation}

\begin{lem}
For $z\ge 2, Q\ge1$ the following estimate holds
\begin{equation}\label{sum of lambda on average}
\sum_{\substack{q\le Q\\q=a^2b}}\left(\sum_{2<n\le z}\lambda_q(n^2-4)-z\frac{\mu(b)}{b}\right)\ll Q^{3/2}\log^2Q.
\end{equation}
\end{lem}
\begin{proof}
It follows from the proof of \cite[Lemma 2.3]{SY} that
\begin{multline*}\label{sum of lambda on average2}
\sum_{q\le Q}\left(\sum_{n\le z}\lambda_q(n^2-4)-z\frac{\mu(b)}{b}\right)=\\
\sum_{q\le Q}\sum_{q_1^2q_2=q}\frac{1}{q_2}\sum_{\substack{k(q_2)\\k\neq0}}S(k^2,1;q_2)
\sum_{n\le z}e\left(k\frac{n}{q_2}\right)+O(Q).
\end{multline*}
Applying the estimate $\sum_{n\le z}e\left(k\frac{n}{q_2}\right)\ll\|k/q_2\|^{-1}$ and Weil's bound \eqref{Weilbound} we obtain
\begin{multline*}
\sum_{q\le Q}\left(\sum_{n\le z}\lambda_q(n^2-4)-z\frac{\mu(b)}{b}\right)\ll\\
\sum_{q_1^2q_2\le Q}q_2^{1/2}\tau_0(q_2)\log q_2\ll Q^{3/2}\log^2Q.
\end{multline*}
\end{proof}


Let $\varphi(x)$ be a smooth function on $[0,\infty)$ such that
\begin{equation*}
\varphi(0)=0,\quad \varphi^{(j)}(x)\ll(1+x)^{-2-\epsilon}, \quad j=0,1,2.
\end{equation*}
Define
\begin{equation}\label{phi0-Def}
\varphi_0=\frac{1}{2\pi}\int_0^{\infty}J_0(y)\varphi(y)dy,
\end{equation}

\begin{equation}\label{phiB-Def}
\varphi_B(x)=\int_0^{1}\int_0^{\infty}\xi xJ_0(\xi x)J_0(\xi y)\varphi(y)dyd\xi,
\end{equation}

\begin{equation}\label{phiHat-Def}
\hat{\varphi}(t)=\frac{\pi i}{2\sinh(\pi t)}\int_0^{\infty}(J_{2it}(x)-J_{-2it}(x))\varphi(x)\frac{dx}{x}.
\footnote{Note that there is a typo (the imaginary unit $i$ is placed in the denominator instead of the numerator) in \cite[p. 68, line -1]{DesIw} and \cite[p. 233, line -3]{LuoSarnakPG} in the definition of $\hat{\varphi}(t)$.}
\end{equation}

In order to avoid convergence problems,  we modify slightly the choice of  $\varphi(x)$ that was made in \cite{DesIw} and \cite{LuoSarnakPG}. For $X,T\gg1$ let
\begin{equation}\label{phi def}
\varphi(x):=\frac{\sinh^2\beta}{2\pi}x^2\exp(ix\cosh\beta)
\end{equation}
with
\begin{equation}\label{beta def}
\beta:=\frac{1}{2}\log X+\frac{i}{2T}.
\end{equation}
It is convenient to introduce the following notation
\begin{equation}\label{c def}
c:=-i\cosh\beta=a-ib,
\end{equation}
\begin{equation}\label{a,b def}
\begin{cases}
a:=\sinh(\log\sqrt{X})\sin((2T)^{-1}),\\
b:=\cosh(\log\sqrt{X})\cos((2T)^{-1}).
\end{cases}
\end{equation}
Note that
\begin{equation}\label{argc}
\arg{c}=-\pi/2+\gamma,\quad 0<\gamma, \quad  T^{-1}\ll\gamma \ll T^{-1}.
\end{equation}

\begin{lem}
For $X,T\gg1, t>0$ the following holds
\begin{multline}\label{phihat expression}
\hat{\varphi}(t)=t\frac{\cosh(\pi t+2i\beta t)}{\sinh(\pi t)}+
\frac{i\sinh(\pi t+2i\beta t)}{2\tanh(\beta)\sinh(\pi t)}=\\
tX^{it}\exp(-t/T)+\frac{i\cosh\beta}{2\sinh\beta}X^{it}\exp(-t/T)+O(\exp(-\pi t)),
\end{multline}
\begin{equation}\label{phi0 expression}
\varphi_0=\frac{-i}{4\pi^2}\left(\frac{2}{\sinh\beta}+\frac{3}{\sinh^3\beta}\right)\ll X^{-1/2},
\end{equation}
\begin{multline}\label{phiBbound}
\varphi_B(x)=
\frac{-i\sinh^2\beta}{2\pi}\int_0^{1}\xi xJ_0(\xi x)\\ \times
\left(\frac{2}{(\cosh^2\beta-\xi^2)^{3/2}}+\frac{3\xi^2}{(\cosh^2\beta-\xi^2)^{5/2}}\right)d\xi
\ll X^{-1/2}\min(x,x^{1/2}).
\end{multline}
\end{lem}
\begin{proof}
The proof is similar to  \cite[Lemma 7]{DesIw}. It follows from \cite[Eq. 6.6621.1]{GR}, \cite[Eq. 15.4.18]{HMF} that
\begin{equation}\label{integral J2it exp}
\int_0^{\infty}J_{2it}(x)\exp(ix\cosh\beta)dx=-\frac{\exp(-(\pi+2i\beta)t)}{i\sinh\beta}.
\footnote{Note that there is the typo in \cite[Eq. 7.7]{DesIw}: the minus sign is missed.}
\end{equation}
Differentiating equation \eqref{integral J2it exp} with respect to $\beta$, we have
\begin{equation}\label{integral J2it x exp}
\int_0^{\infty}J_{2it}(x)\exp(ix\cosh\beta)xdx=-\exp(-(\pi+2i\beta)t)\left(
\frac{2it}{\sinh^2\beta}+\frac{\cosh\beta}{\sinh^3\beta}
\right).
\end{equation}
Substituting \eqref{phi def} to \eqref{phiHat-Def} and using \eqref{integral J2it x exp}, we obtain \eqref{phihat expression}.

Let us now prove \eqref{phi0 expression}. Differentiating equation \eqref{integral J2it exp} with respect to $\beta$ and taking $t=0$ yields
\begin{equation}\label{integral J0 exp}
\int_0^{\infty}J_{0}(x)\exp(ix\cosh\beta)xdx=-\frac{\cosh\beta}{\sinh\beta}.
\footnote{Correcting the typo in \cite[Eq. 7.7]{DesIw},  the formula \eqref{integral J0 exp} has an additional minus sign compared to \cite[Eq. 7.8]{DesIw}.}
\end{equation}
 Differentiating equation \eqref{integral J2it exp} with respect to $\beta$ twice,  taking $t=0$ and using \eqref{integral J0 exp}, we show that
\begin{equation}\label{integral J0 x2 exp}
\int_0^{\infty}J_{0}(x)\exp(ix\cosh\beta)x^2dx=-i\left(
\frac{2}{\sinh^3\beta}+\frac{3}{\sinh^5\beta}
\right).
\end{equation}
Substituting \eqref{phi def} to \eqref{phi0-Def} and using \eqref{integral J0 x2 exp}, we prove \eqref{phi0 expression}.

Finally, the first equality in \eqref{phiBbound} can be proved similarly to \cite[Eq. 7.5]{DesIw}. The only difference is that we now use \eqref{integral J0 x2 exp} instead of \eqref{integral J0 exp}. The final estimate on $\varphi_B(x)$ in \eqref{phiBbound} can be proved in the same way as \cite[Lemma 11]{DesIw}.
\end{proof}
Now we are ready to prove \eqref{PrG1 =A+B+O}. As it was mentioned in the introduction, our arguments are similar to Iwaniec's proof of \eqref{PrG =A+B+O} in \cite{IwPG} (see also \cite[Section 6]{LuoSarnakPG} for more details). For completeness we sketch the proof below.

Consider the following  sum
\begin{equation}\label{spec.sum averaged over n}
\sum_n h(n)\sum_{j}\frac{|\rho_j(n)|^2}{\cosh(\pi t_j)}\hat{\varphi}(t_j),
\end{equation}
where $h(x)$ is as in Theorem \ref{mainthm} and $\rho_j(n)$ is the $n$-th Fourier coefficient of the Hecke-Maa{\ss} cusp forms. Applying the Mellin inversion formula to $h(n)$ and using the properties of the Rankin-Selberg $L$-function, we have
\begin{equation}\label{spec.sum decomposition0}
\sum_{j}\hat{\varphi}(t_j)=
\frac{\zeta(2)}{2N}\sum_n h(n)\sum_{j}\frac{|\rho_j(n)|^2}{\cosh(\pi t_j)}\hat{\varphi}(t_j)+\mathcal{B}_1,
\end{equation}
where $\mathcal{B}_1$ is defined by \eqref{B1 def}.  Applying the Kuznetsov trace formula \cite[p. 234]{LuoSarnakPG} to the sum over $j$ on the right-hand side of equation \eqref{spec.sum decomposition0} and estimating some of the arising terms using \eqref{phi0 expression} and \eqref{phiBbound} in the same way as in the paper of Luo-Sarnak  \cite[p. 234]{LuoSarnakPG}, we prove \eqref{PrG1 =A+B+O}.


\section{Sums of Kloosterman sums}
In order to analyse sums of Kloosterman sums in \eqref{A1 def} we apply the following result of Kuznetsov \cite{K}.
Let
\begin{equation}\label{Zpsi(s) def}
Z_{\psi}(s):=\sum_{q=1}^{\infty}\frac{1}{q}\sum_{n=1}^{\infty}\frac{S(n,n;q)}{n^s}
\psi\left(\frac{4\pi n}{q}\right).
\end{equation}

\begin{lem}\label{Kuznetsov lemma}
Assume that for $\Delta>3/4$ we have
\begin{equation*}\psi(x)\ll x^{2\Delta} \quad \text{as }x\to+0,
\end{equation*}
\begin{equation*}\psi(x)\ll 1 \quad \text{as }x\to+\infty.
\end{equation*} Then for $1+2\Delta>\Re{s}>3/2$
\begin{equation}\label{Zpsi(s) equation}
Z_{\psi}(s)=\frac{2\zeta(s)}{\zeta(2s)}\mathop{{\sum}^*}_{n=0}^{\infty}\mathscr{L}_{n^2-4}(s)\Psi(n,s),
\end{equation}
with
\begin{equation}\label{Psi(n,s) def}
\Psi(n,s)=(4\pi)^{s-1}\int_0^{\infty}\psi(x)\cos\left(\frac{nx}{2}\right)x^{-s}dx,
\end{equation}
provided that the function $\psi(x)$ is such that the series on the right-hand side of \eqref{Zpsi(s) equation} is absolutely convergent.
\end{lem}
\begin{rem*}
Lemma \ref{Kuznetsov lemma} was proved by Kuznetsov in \cite{K}.  We provide a proof below because this reference is hard to find.
\end{rem*}
\begin{proof}
Using the conditions on $\psi(x)$ and Weil's bound  \eqref{Weilbound}, we find that  the series \eqref{Zpsi(s) def} is  absolutely convergent for $\Re{s}>3/2.$

On the one hand, we have
\begin{equation}\label{Zpsi(s) equation1}
Z_{\psi}(s)=\sum_{q=1}^{\infty}\frac{1}{q}\sum_{l=1}^{q}S(l,l;q)
\sum_{n=0}^{\infty}(l+nq)^{-s}
\psi\left(\frac{4\pi(l+nq)}{q}\right).
\end{equation}
On the other hand, since $S(-l,-l;q)=S(l,l;q)$, it follows that
\begin{equation}\label{Zpsi(s) equation2}
Z_{\psi}(s)=\sum_{q=1}^{\infty}\frac{1}{q}\sum_{l=0}^{q-1}S(l,l;q)
\sum_{n=0}^{\infty}(q-l+nq)^{-s}
\psi\left(\frac{4\pi(q-l+nq)}{q}\right).
\end{equation}
Let
\begin{equation}\label{tilde psi def}
F_{\psi}\left(x,s\right)=\sum_{n=0}^{\infty}(n+x)^{-s}
\psi\left(4\pi(n+x)\right).
\end{equation}
Then \eqref{Zpsi(s) equation1} and \eqref{Zpsi(s) equation2} can be rewritten as
\begin{equation}\label{Zpsi(s) equation1.1}
Z_{\psi}(s)=\sum_{q=1}^{\infty}\frac{1}{q^{1+s}}\sum_{l=1}^{q}S(l,l;q)F_{\psi}\left(\frac{l}{q},s\right),
\end{equation}
\begin{equation}\label{Zpsi(s) equation2.1}
Z_{\psi}(s)=\sum_{q=1}^{\infty}\frac{1}{q^{1+s}}\sum_{l=0}^{q-1}S(l,l;q)
F_{\psi}\left(1-\frac{l}{q},s\right).
\end{equation}
Let us assume that $3/2<\Re{s}<2\Delta$. Then $$\lim_{x\to0}\psi(x)x^{-s}=0,$$ and consequently
$
F_{\psi}\left(1,s\right)=F_{\psi}\left(0,s\right).
$
Therefore, \eqref{Zpsi(s) equation2.1} can be written as
\begin{equation}\label{Zpsi(s) equation2.2}
Z_{\psi}(s)=\sum_{q=1}^{\infty}\frac{1}{q^{1+s}}\sum_{l=1}^{q}S(l,l;q)
F_{\psi}\left(1-\frac{l}{q},s\right).
\end{equation}
Applying \eqref{Zpsi(s) equation1.1} and \eqref{Zpsi(s) equation2.2} we have
\begin{equation}\label{Zpsi(s) equation3}
Z_{\psi}(s)=\frac{1}{2}\sum_{q=1}^{\infty}\frac{1}{q^{1+s}}\sum_{l=1}^{q}S(l,l;q)\left(
F_{\psi}\left(\frac{l}{q},s\right)+F_{\psi}\left(1-\frac{l}{q},s\right)
\right).
\end{equation}
Expanding the function $F_{\psi}(x,s)+F_{\psi}(1-x,s)$ in the Fourier series gives
\begin{equation}\label{tilde psi Fourier}
\frac{F_{\psi}(x,s)+F_{\psi}(1-x,s)}{2}=2\mathop{{\sum}^*}_{n=0}^{\infty}
\Psi(n,s)\cos(2\pi nx).
\end{equation}
Substituting \eqref{tilde psi Fourier} in \eqref{Zpsi(s) equation3} yields
\begin{equation}\label{Zpsi(s) equation4}
Z_{\psi}(s)=2\sum_{q=1}^{\infty}\frac{1}{q^{1+s}}\mathop{{\sum}^*}_{n=0}^{\infty}
\Psi(n,s)\sum_{l=1}^{q}S(l,l;q)\cos\left(2\pi\frac{ln}{q}\right).
\end{equation}
The change of the orders of summation is justified due to the absolute convergence of the right-hand side of \eqref{Zpsi(s) equation4}. This follows from Weil's bound \eqref{Weilbound} and the conditions of the lemma, namely that $\Re{s}>3/2$ and that the function $\Psi(n,s)$ is of rapid decay.

The last step of the proof is to show that
\begin{equation}\label{sum Kloosterman sum to pho}
\sum_{l=1}^{q}S(l,l;q)\cos\left(2\pi\frac{ln}{q}\right)=q\rho_q(n^2-4).
\end{equation}
Since Kloosterman sums are always real we obtain
\begin{multline*}
\sum_{l=1}^{q}S(l,l;q)\cos\left(2\pi\frac{ln}{q}\right)=\Re{\left(
\sum_{\substack{d,e=1\\de\equiv1(q)}}^q
\sum_{l=1}^{q}e\left(\frac{l(d+e+n)}{q}\right)\right)}=\\
q\sum_{\substack{d=1\\d+d^{*}+n\equiv0(q)}}^q1=
q\rho_q(n^2-4),
\end{multline*}
where $dd^{*}\equiv1(q).$ For the proof of the last equality see \cite[Lemma 2.3]{SY}.
\end{proof}

\begin{lem}\label{Kuznetsov lemma application}
The following exact formula holds
\begin{multline}\label{sum Kloosterman decomposition}
\sum_{n}h(n)\sum_{q=1}^{\infty}\frac{S(n,n;q)}{q}
\varphi\left(\frac{4\pi n}{q}\right)=\frac{2\tilde{h}(1)}{\zeta(2)}
\mathop{{\sum}^*}_{\substack{n=0\\n\neq2}}^{\infty}\mathscr{L}_{n^2-4}(1)\Phi(n,1)+\\
2\res_{s=1}\left(\tilde{h}(s)\frac{\zeta(s)\zeta(2s-1)}{\zeta(2s)}\Phi(2,s)\right)+\\
\frac{1}{2\pi i}
\int_{(1/2)}\tilde{h}(s)\frac{2\zeta(s)}{\zeta(2s)}\mathop{{\sum}^*}_{n=0}^{\infty}\mathscr{L}_{n^2-4}(s)\Phi(n,s)ds,
\end{multline}
where
\begin{equation}\label{Phi(n,s)def}
\Phi(n,s)=\frac{\sinh^2\beta}{2\pi}\frac{(4\pi)^{s-1}\Gamma(3-s)}{(c^2+n^2/4)^{3/2-s/2}}\cos\left((3-s)\arctan\frac{n}{2c}\right),
\end{equation}
\end{lem}
\begin{proof}
With the goal of using Lemma \ref{Kuznetsov lemma} in order to evaluate \eqref{A1 def}, we  apply the Mellin inversion formula to the function $h(n)$, getting
\begin{equation*}
\sum_{n}h(n)\sum_{q=1}^{\infty}\frac{S(n,n;q)}{q}
\varphi\left(\frac{4\pi n}{q}\right)=\frac{1}{2\pi i}\int_{(\sigma)}\tilde{h}(s)Z_{\varphi}(s)ds,
\end{equation*}
where $3>\sigma>3/2.$  The function $\varphi(x)$ defined by \eqref{phi def} satisfies the conditions of  Lemma \ref{Kuznetsov lemma} because it behaves like $x^2$ when $x\to 0$ and it decays exponentially when $x\to+\infty$. Then Lemma \ref{Kuznetsov lemma} yields
\begin{multline}\label{sum of Kloosterman2}
\sum_{n}h(n)\sum_{q=1}^{\infty}\frac{S(n,n;q)}{q}
\varphi\left(\frac{4\pi n}{q}\right)=\\
\frac{1}{2\pi i}\int_{(\sigma)}
\frac{2\zeta(s)\tilde{h}(s)}{\zeta(2s)}\mathop{{\sum}^*}_{n=0}^{\infty}\mathscr{L}_{n^2-4}(s)\Phi(n,s)ds,
\end{multline}
where
\begin{multline}\label{Phin integral}
\Phi(n,s)=(4\pi)^{s-1}\int_0^{\infty}\varphi(x)\cos\left(\frac{nx}{2}\right)x^{-s}dx=\\
\frac{\sinh^2\beta}{2\pi}(4\pi)^{s-1}\int_0^{\infty}\cos\left(\frac{nx}{2}\right)\exp(ix\cosh\beta)x^{2-s}dx.
\end{multline}
To evaluate \eqref{Phin integral} we apply \cite[Eq. 3.944.6]{GR} and obtain \eqref{Phi(n,s)def}.  As explained in Section \ref{section:notation}, the function $\mathscr{L}_{n^2-4}(s)$ has a pole at $s=1$ only if $n=2$. In the later case $\mathscr{L}_{0}(s)=\zeta(2s-1).$ Finally, moving the line of integration in \eqref{sum of Kloosterman2} to $\Re{s}=1/2$ we obtain \eqref{sum Kloosterman decomposition}.
\end{proof}


To derive an asymptotic formula from \eqref{sum Kloosterman decomposition} we need the following lemma.
\begin{lem}\label{estimate on exp factor}
Let $c=a-ib$ being defined in \eqref{c def} and
\begin{equation}\label{z_pm def}
z_{\pm}(n):=2ci\pm n=2b\pm n+2ai.
\end{equation}
Then for $n\ge0$ and any real $t$ the following inequality holds
\begin{equation}\label{exp factor2}
-\pi|t|-t\arg{(n^2/4+c^2)}\pm t(\arg{z_{+}(n)}-\arg{z_{-}(n)})\le0.
\end{equation}
\end{lem}
\begin{proof}
There are several cases to consider due to the presence of $\pm$ and $t\lessgtr0.$  For $t<0$ it is required to prove that
\begin{equation}\label{exp factor t<0 1}
\arg{(n^2/4+c^2)}\pm (\arg{z_{+}(n)}-\arg{z_{-}(n)})\le\pi.
\end{equation}

Let us consider the $-$ case in \eqref{exp factor t<0 1}. Then we need to show that
\begin{equation}\label{exp factor t<0 2}
\arg{(n^2/4+c^2)}- \arg{z_{+}(n)}+\arg{z_{-}(n)}\le\pi
\end{equation}
This is satisfied because $$\arg{z_{-}(n)}<\pi,\quad \arg{z_{+}(n)}>0, \quad \arg{(n^2/4+c^2)}<0.$$

Let us consider the $+$ case in \eqref{exp factor t<0 1}. Then it is required to prove that
\begin{equation}\label{exp factor t<0 2}
\arg{(n^2/4+c^2)}+ \arg{z_{+}(n)}-\arg{z_{-}(n)}\le\pi.
\end{equation}
This inequality holds because $$\arg{z_{-}(n)}>0,\quad \arg{z_{+}(n)}\ll T^{-1}, \quad \arg{(n^2/4+c^2)}<0.$$  As a result,
\eqref{exp factor2} is proved for $t<0$.

Now we assume that  $t>0$.  Then we need to show that
\begin{equation}\label{exp factor t>0 1}
-\arg{(n^2/4+c^2)}\pm (\arg{z_{+}(n)}-\arg{z_{-}(n)})\le\pi.
\end{equation}
If we consider the $+$ case in \eqref{exp factor t>0 1}, then our goal is to prove that
\begin{equation}\label{exp factor t>0 2}
-\arg{(n^2/4+c^2)}+\arg{z_{+}(n)}-\arg{z_{-}(n)}\le\pi,
\end{equation}
or equivalently,
\begin{multline}\label{exp factor t>0 3}
\arg{(n^2/4+a^2-b^2+2abi)}+\arg{(2b+n+2ai)}\\-\arg{(2b-n+2ai)}\le\pi.
\end{multline}

If $n>2(b^2-a^2)^{1/2}$, then
$$\arg{(n^2/4+a^2-b^2+2abi)}<\pi/2,\quad \arg{z_{+}(n)}\ll T^{-1},\quad
\arg{z_{-}(n)}>0,$$ and consequently,  \eqref{exp factor t>0 3} is satisfied.

So we are left to analyse the case $n\le2(b^2-a^2)^{1/2}.$  In this case
\begin{equation}\label{arg(n^2+c^2/4)}
\arg{(n^2/4+a^2-b^2+2abi)}=\frac{\pi}{2}+\arctan\frac{b^2-a^2-n^2/4}{2ab}.
\end{equation}
Furthermore,
\begin{equation*}
\arg{(2b+n+2ai)}=\arctan\frac{2a}{2b+n},
\end{equation*}
\begin{equation*}
\arg{(2b-n+2ai)}=\arctan\frac{2a}{2b-n}.
\end{equation*}
Therefore, in order to prove \eqref{exp factor t>0 3} it is sufficient to show that
\begin{equation}\label{exp factor t>0 4}
\arctan\frac{b^2-a^2-n^2/4}{2ab}+\arctan\frac{2a}{2b+n}\le \pi/2+\arctan\frac{2a}{2b-n}.
\end{equation}
Let us prove that the left-hand side of \eqref{exp factor t>0 4} is bounded by $\pi/2.$ With this goal, we evaluate the tangent of the left-hand side of \eqref{exp factor t>0 4}, which turns out to be positive because
$$
\frac{b^2-a^2-n^2/4}{2ab}\times\frac{2a}{2b+n}<1.
$$
Thus \eqref{exp factor t>0 4} is proved.

The $+$ case in \eqref{exp factor t>0 1} requires proving that
\begin{equation}\label{exp factor3}
-\arg{(n^2/4+c^2)}-\arg{z_{+}(n)}+\arg{z_{-}(n)}\le\pi,
\end{equation}
or equivalently,
\begin{multline}\label{exp factor4}
\arg{(n^2/4+a^2-b^2+2abi)}-\arg{(2b+n+2ai)}\\+\arg{(2b-n+2ai)}\le\pi.
\end{multline}

First, assume that $n\le2(b^2-a^2)^{1/2}.$ Using \eqref{arg(n^2+c^2/4)} and
\begin{equation*}
\arg{(2b-n+2ai)}=\frac{\pi}{2}-\arctan\frac{2b-n}{2a},
\end{equation*}
the inequality \eqref{exp factor4} can be written as
\begin{equation}\label{exp factor5}
\arctan\frac{b^2-a^2-n^2/4}{2ab}\le \arctan\frac{2b-n}{2a}+\arctan\frac{2a}{2b+n}.
\end{equation}
This is equivalent to
\begin{equation}\label{exp factor6}
\frac{b^2-a^2-n^2/4}{2ab}\le\frac{(2b-n)/(2a)+2a/(2b+n)}{1-(2b-n)/(2b+n)}.
\end{equation}
Simplifying we obtain
\begin{equation}\label{exp factor7}
(b^2-a^2-n^2/4)(2n-4b)\le 8a^2b.
\end{equation}
Inequality \eqref{exp factor7} always holds since the left-hand side is negative. This yields \eqref{exp factor4}.

Second, assume that  $2(b^2-a^2)^{1/2}<n\le2b.$ Then \eqref{exp factor4} can be rewritten as
\begin{equation*}
\arctan\frac{2ab}{n^2/4-b^2+a^2}+\arctan\frac{2a}{2b-n}-\arctan\frac{2a}{2b+n}\le\pi.
\end{equation*}
This can be simplified to
\begin{equation}\label{exp factor8}
-\arctan\frac{n^2/4-b^2+a^2}{2ab}-\arctan\frac{2b-n}{2a}-\arctan\frac{2a}{2b+n}\le0.
\end{equation}
The inequality \eqref{exp factor8} is always satisfied  and so is the inequality  \eqref{exp factor4}.

Third, assume that $n>2b.$ In this case \eqref{exp factor4} can be formulated in the following form
\begin{equation}\label{exp factor9}
\arctan\frac{2ab}{n^2/4-b^2+a^2}\le\arctan\frac{2a}{n-2b}+\arctan\frac{2a}{2b+n},
\end{equation}
which is true because
\begin{equation}\label{exp factor10}
\arctan\frac{2ab}{n^2/4-b^2+a^2}+\arctan\frac{2a}{2b+n}=\arctan\frac{2a}{n-2b}.
\end{equation}
The last equation  can be proved by taking the tangent of both sides. Therefore, inequality \eqref{exp factor4} is satisfied.

\end{proof}

\begin{lem}\label{lemma estimates on the first error}
For $N,X,T\gg1$ the following asymptotic formula holds
\begin{multline}\label{sum Kloosterman decomposition2}
\sum_{n}h(n)\sum_{q=1}^{\infty}\frac{S(n,n;q)}{q}
\varphi\left(\frac{4\pi n}{q}\right)=\frac{2\tilde{h}(1)}{\zeta(2)}
\sum_{n=3}^{\infty}\mathscr{L}_{n^2-4}(1)\Phi(n,1)\\+
O\left(N\log(NX)\right)+O\left(N^{1/2}X^{1/4+\theta}T^{3/2}\left(1+\frac{T}{X^{1/2}}\right)\right),
\end{multline}
where for $n>0$
\begin{equation}\label{Phi(n,1)def}
\Phi(n,1)=\frac{\sinh^2\beta}{2\pi c^2}
\frac{1-n^2/(4c^2)}{(1+n^2/(4c^2))^{2}}.
\end{equation}
\end{lem}
\begin{proof}
According to \eqref{beta def}, \eqref{c def}  we have $|\sinh^2\beta|,|c|^2\sim X$. Then it follows from \eqref{Phi(n,s)def} that for $n=0,1,2$
$$|\Phi(n,1)|\ll 1,\quad
|\Phi'(2,1)|\ll \log X,
$$
where the derivative is taken with respect to $s.$ Since $\tilde{h}(1)\ll N$ and $\tilde{h}'(1)\ll N\log N$ we obtain
\begin{multline*}
\tilde{h}(1)\left(\mathscr{L}_{-4}(1)\Phi(0,1)+\mathscr{L}_{-3}(1)\Phi(1,1)\right)+\\
2\res_{s=1}\left(\tilde{h}(s)\frac{\zeta(s)\zeta(2s-1)}{\zeta(2s)}\Phi(2,s)\right)=O\left(N\log(NX)\right).
\end{multline*}
Thus, it remains to estimate the integral on the right-hand side of \eqref{sum Kloosterman decomposition}. Note that $$\tilde{h}(1/2+it)\ll N^{1/2}(1+|t|)^{-A}\text{ for any }A>0.$$ Applying \eqref{eq:subconvexity}, it is sufficient to prove  the following estimate
\begin{equation}\label{sum of Phi(n,1/2+it)0}
\sum_{n=0}^{\infty}(1+n)^{2\theta}|\Phi(n,1/2+it)|\ll(1+|t|)^{B}
X^{1/4+\theta}T^{3/2}\left(1+\frac{T}{X^{1/2}}\right).
\end{equation}
As a consequence of \eqref{Phi(n,s)def}, \eqref{argc} and the Stirling formula we obtain
\begin{multline}\label{sum of Phi(n,1/2+it)}
|\Phi(0,1/2+it)|\ll(1+|t|)^{B}\frac{X}{|c|^{5/2}}\exp(-\pi|t|/2-t\arg c)\\ \ll
(1+|t|)^{B}X^{-1/4},
\end{multline}
since $|c|^2\sim X.$ Hence it remains to prove that
\begin{equation}\label{sum of Phi(n,1/2+it)}
\sum_{n=1}^{\infty}n^{2\theta}|\Phi(n,1/2+it)|\ll(1+|t|)^{B}
X^{1/4+\theta}T^{3/2}\left(1+\frac{T}{X^{1/2}}\right)
\end{equation}
for some constant $B>0$. Using \eqref{Phi(n,s)def} we obtain
\begin{multline}\label{Phi(n,1/2+it)estimate1}
|\Phi(n,1/2+it)|\ll\frac{(1+|t|)^2X\exp(-\pi|t|/2)}{|n^2/4+c^2|^{5/4}}\times\\
\left|(n^2/4+c^2)^{it/2}\cos\left((5/2-it)\arctan\frac{n}{2c}\right)\right|.
\end{multline}
It follows from  \cite[Eq. 4.23.26]{HMF} that
\begin{equation*}
\arctan\frac{n}{2c}=\frac{i}{2}\log\frac{2ci+n}{2ci-n}=\frac{i}{2}\left(\log z_{+}(n)-\log z_{-}(n)\right),
\end{equation*}
where for $c=a-ib$ we have
\begin{equation*}
z_{\pm}(n)=2ci\pm n=2b\pm n+2ai.
\end{equation*}
Therefore,
\begin{multline}\label{cos estimate}
\left|\cos\left((5/2-it)\arctan\frac{n}{2c}\right)\right|\ll\\
\sum_{j=\pm1}\left|\frac{z_{+}(n)}{z_{-}(n)}\right|^{5j/4}
\exp\left(jt(\arg{z_{+}(n)}-\arg{z_{-}(n)})/2\right).
\end{multline}
Substituting \eqref{cos estimate} to \eqref{Phi(n,1/2+it)estimate1} gives
\begin{multline}\label{Phi(n,1/2+it)estimate2}
|\Phi(n,1/2+it)|\ll\frac{(1+|t|)^2X}{|n^2/4+c^2|^{5/4}}
\sum_{j=\pm1}\left|\frac{z_{+}(n)}{z_{-}(n)}\right|^{5j/4}\times\\
\exp\left(-\pi|t|/2-t\arg{(n^2/4+c^2)}/2+jt(\arg{z_{+}(n)}-\arg{z_{-}(n)})/2\right).
\end{multline}
By Lemma \ref{estimate on exp factor} for $j=\pm1$ we obtain
\begin{equation}\label{exp factor}
-\pi|t|/2-t\arg{(n^2/4+c^2)}/2+jt(\arg{z_{+}(n)}-\arg{z_{-}(n)})/2\le0.
\end{equation}
According to  \eqref{argc} we have
\begin{multline}\label{z+abs value}
|z_{+}(n)|^2=(2|c|\cos\gamma+n)^2+(2|c|\sin\gamma)^2=\\(2|c|)^2\left(\left(1-\frac{n}{2|c|}\right)^2+\frac{4n}{2|c|}\cos^2\frac{\gamma}{2}\right),
\end{multline}
\begin{equation}\label{z-abs value}
|z_{-}(n)|^2=(2|c|)^2\left(\left(1-\frac{n}{2|c|}\right)^2+\frac{4n}{2|c|}\sin^2\frac{\gamma}{2}\right),
\end{equation}
\begin{multline}\label{n^2/4+c^2-abs value}
|n^2/4+c^2|^2=(n^2/4-|c|^2\cos2\gamma)^2+(|c|^2\sin2\gamma)^2=\\
|c|^4\left(\left(1-\left(\frac{n}{2|c|}\right)^2\right)^2+4\left(\frac{n}{2|c|}\right)^2\sin^2\gamma\right).
\end{multline}
It follows from \eqref{z+abs value}, \eqref{z-abs value}, \eqref{n^2/4+c^2-abs value} and the fact that $|c|^2\sim X$ that
\begin{equation}\label{Phi(n,1/2+it)estimate3}
n^{2\theta}|\Phi(n,1/2+it)|\ll(1+|t|)^2X^{-1/4+\theta}\sum_{j=\pm1}f_j\left(\frac{n}{2|c|}\right),
\end{equation}
where for $x>0$
\begin{equation}\label{fdef}
f_j(x)=\left(\frac{(1-x)^2+4x\cos^2\gamma/2}{(1-x)^2+4x\sin^2\gamma/2}\right)^{5j/8}
\frac{x^{2\theta}}{\left((1-x^2)^2+4x^2\sin^2\gamma\right)^{5/8}}.
\end{equation}
Note that since $0<\gamma\ll T^{-1}$ (see \eqref{argc}), we have $f_{-1}(x)<f_1(x).$ Note that the function $f_1(x)$ is continuous and has finitely many monotonic segments. As a result,
\begin{multline}\label{Phi(n,1/2+it)estimate4}
\sum_{n=1}^{\infty}n^{2\theta}|\Phi(n,1/2+it)|\ll(1+|t|)^2X^{-1/4+\theta}\times\\
\left(X^{1/2}\int_0^{\infty}f_1(x)dx+\max_{x>0}|f_1(x)|\right).
\end{multline}
Since $T^{-1}\ll\gamma\ll T^{-1},$ for $x>0$ we have
\begin{equation}\label{f1 estimate1}
f_1(x)\ll\min\left(T^{5/2},\frac{x^{2\theta}}{|1-x|^{5/2}}\right).
\end{equation}
Finally, using \eqref{f1 estimate1} to estimate the right-hand side of \eqref{Phi(n,1/2+it)estimate4} we obtain
\begin{equation}\label{Phi(n,1/2+it)estimate5}
\sum_{n=1}^{\infty}n^{2\theta}|\Phi(n,1/2+it)|\ll
(1+|t|)^{2}
X^{1/4+\theta}T^{3/2}\left(1+\frac{T}{X^{1/2}}\right).
\end{equation}
\end{proof}


\begin{lem}\label{lemma estimate on the main error}
For $N,X,T\gg1$  the following estimate holds
\begin{equation}\label{estimate on sum L(1)}
\sum_{n=3}^{\infty}\mathscr{L}_{n^2-4}(1)\Phi(n,1)\ll
\max(X^{1/4+\theta/2}T^{3/2},X^{\theta/2}T^2)\log^2X.
\end{equation}
\end{lem}
\begin{proof}
Applying \eqref{approximate func.eq.} we obtain
\begin{multline}\label{application of approx fun.eq}
\sum_{n=3}^{\infty}\mathscr{L}_{n^2-4}(1)\Phi(n,1)=
\sum_{n=3}^{\infty}\Phi(n,1)S_V(n^2-4)-\\
\frac{1}{2\pi i}\int_{(-1/2)}\sum_{n=3}^{\infty}\Phi(n,1)\mathscr{L}_{n^2-4}(1+s)V^s\Gamma(s)ds.
\end{multline}
Let us first estimate the integral. Using \eqref{eq:subconvexity} and \eqref{Phi(n,1)def} we have
\begin{multline}\label{estimate on the error in AFE}
\frac{1}{2\pi i}\int_{(-1/2)}\sum_{n=3}^{\infty}\Phi(n,1)\mathscr{L}_{n^2-4}(1+s)V^s\Gamma(s)ds\ll\\
V^{-1/2}\sum_{n=3}^{\infty}|\Phi(n,1)|n^{2\theta}\ll
V^{-1/2}\sum_{n=3}^{\infty}\left|
\frac{1-n^2/(4c^2)}{(1+n^2/(4c^2))^{2}}\right|n^{2\theta}.
\end{multline}
It follows from \eqref{argc} that
\begin{multline}\label{Phi(x,1) abs.value1}
\left|1-\frac{n^2}{4c^2}\right|^2=1+\left(\frac{n^2}{4|c|^2}\right)^2+2\frac{n^2}{4|c|^2}\cos(2\gamma)\\=
\left(1-\frac{n^2}{4|c|^2}\right)^2+4\frac{n^2}{4|c|^2}\cos^2\gamma,
\end{multline}
\begin{equation}\label{Phi(x,1) abs.value2}
\left|1+\frac{n^2}{4c^2}\right|^2=
\left(1-\frac{n^2}{4|c|^2}\right)^2+4\frac{n^2}{4|c|^2}\sin^2\gamma.
\end{equation}

Therefore,
\begin{equation}\label{estimate on sum over n 1}
\sum_{n=3}^{\infty}\left|
\frac{1-n^2/(4c^2)}{(1+n^2/(4c^2))^{2}}\right|n^{2\theta}\ll
X^{\theta}\sum_{n=3}^{\infty}g\left(\frac{n}{2|c|}\right),
\end{equation}
where
\begin{equation}\label{g def}
g(y)=
\frac{((1-y^2)^2+4y^2\cos^2\gamma)^{1/2}}{(1-y^2)^2+4y^2\sin^2\gamma}y^{2\theta}.
\end{equation}
 Since the function $g(x)$ is continuous and has finitely many monotonic segments, we obtain
\begin{equation}\label{estimate on sum over n 2}
\sum_{n=3}^{\infty}\left|
\frac{1-n^2/(4c^2)}{(1+n^2/(4c^2))^{2}}\right|n^{2\theta}\ll
X^{\theta}
\left(X^{1/2}\int_0^{\infty}g(y)dy+\max_{y>0}|g(y)|\right).
\end{equation}
Note that we used the fact that $|c|\sim X^{1/2}.$ Since $T^{-1}\ll\gamma\ll T^{-1},$ for $y>0$ the following  holds
\begin{equation}\label{g estimate1}
g(y)\ll\min\left(T^{2},\frac{y^{2\theta}}{|1-y|^{2}}\right).
\end{equation}
Using \eqref{g estimate1} to estimate the right-hand side of \eqref{estimate on sum over n 2}, we obtain
\begin{equation}\label{estimate on sum over n 3}
\sum_{n=3}^{\infty}\left|
\frac{1-n^2/(4c^2)}{(1+n^2/(4c^2))^{2}}\right|n^{2\theta}\ll
X^{\theta}T^2+X^{1/2+\theta}T.
\end{equation}
Applying \eqref{estimate on sum over n 1}, \eqref{estimate on the error in AFE} and \eqref{application of approx fun.eq} we prove
\begin{multline}\label{application of approx fun.eq2}
\sum_{n=3}^{\infty}\mathscr{L}_{n^2-4}(1)\Phi(n,1)=
\sum_{n=3}^{\infty}\Phi(n,1)S_V(n^2-4)+\\
O\left(V^{-1/2}\left(X^{\theta}T^2+X^{1/2+\theta}T\right)\right).
\end{multline}
To estimate the sum in \eqref{application of approx fun.eq2} we apply \eqref{SV def} and   obtain
\begin{equation}\label{application of approx fun.eq2.1}
\sum_{n=3}^{\infty}\Phi(n,1)S_V(n^2-4)=
\sum_{q=1}^{\infty}\frac{\exp(-q/V)}{q}
\sum_{n=0}^{\infty}\Phi(n,1)\lambda_q(n^2-4)a(n),
\end{equation}
where $a(n)=0$ for $n=0,1,2$ and $a(n)=1$ for $n>2.$ Abel's summation formula yields
\begin{equation}\label{application of approx fun.eq2.2}
\sum_{n=0}^{\infty}\Phi(n,1)\lambda_q(n^2-4)a(n)=
-\int_0^{\infty}\Phi'(x,1)\sum_{0\le n\le x}\lambda_q(n^2-4)a(n)dx.
\end{equation}
We remark that for a real positive $x$ the function $\Phi(x,1)$ is still defined by \eqref{Phi(n,1)def}.
Let $q=a^2b$. Then
\begin{multline}\label{application of approx fun.eq2.3}
\sum_{n=0}^{\infty}\Phi(n,1)\lambda_q(n^2-4)a(n)=
-\int_0^{\infty}\Phi'(x,1)x\frac{\mu(b)}{b}dx-\\
\int_0^{\infty}\Phi'(x,1)\left(\sum_{0\le n\le x}\lambda_q(n^2-4)a(n)-x\frac{\mu(b)}{b}\right)dx=\\
\frac{\mu(b)}{b}\int_0^{\infty}\Phi(x,1)dx-
\int_0^{\infty}\Phi'(x,1)\left(\sum_{0\le n\le x}\lambda_q(n^2-4)a(n)-x\frac{\mu(b)}{b}\right)dx.
\end{multline}

Applying  Abel's summation formula, \eqref{sum of lambda on average} and
\cite[Eq. 4.358.2]{GR}, we have
\begin{equation}\label{application of approx fun.eq2.4}
\sum_{\substack{q=1\\q=a^2b}}^{\infty}\frac{\exp(-q/V)}{q}\left(\sum_{n\le z}\lambda_q(n^2-4)a(n)-z\frac{\mu(b)}{b}\right)\ll V^{1/2}\log^2V.
\end{equation}
Substituting \eqref{application of approx fun.eq2.3} to \eqref{application of approx fun.eq2.1} and using \eqref{application of approx fun.eq2.4}, we obtain
\begin{multline}\label{application of approx fun.eq3}
\sum_{n=3}^{\infty}\Phi(n,1)S_V(n^2-4)=\sum_{\substack{q=1\\q=a^2b}}^{\infty}\frac{\exp(-q/V)\mu(b)}{qb}\int_0^{\infty}\Phi(x,1)dx+\\
O\left(V^{1/2}\log^2V\int_0^{\infty}|\Phi'(x,1)|dx\right).
\end{multline}
It follows from \eqref{Phi(n,1)def} that
\begin{equation}\label{Phi(x,1)def}
\Phi(x,1)=\frac{\sinh^2\beta}{2\pi c^2}
\frac{1-x^2/(4c^2)}{(1+x^2/(4c^2))^{2}}.
\end{equation}
Hence using the fact that $|\sinh^2\beta|\sim|c|^2\sim X$, we show the following inequality
\begin{multline}\label{Phi(x,1) derivative abs.value}
|\Phi'(x,1)|\ll\frac{x}{|c|^2}\frac{|3/2-(x/(2c))^2|}{|1+(x/(2c))^2|^3}=\\
\frac{x}{|c|^2}\left(\left(\frac{3}{2}-\left(\frac{x}{2|c|}\right)^2\right)^2+6\left(\frac{x}{2|c|}\right)^2\cos^2\gamma\right)^{1/2}\times\\
\left(\left(1-\left(\frac{x}{2|c|}\right)^2\right)^2+4\left(\frac{x}{2|c|}\right)^2\sin^2\gamma\right)^{-3/2}.
\end{multline}
Integrating \eqref{Phi(x,1) derivative abs.value} with respect to $x$ gives
\begin{multline}\label{|Phi'(x,1)| integral}
\int_0^{\infty}|\Phi'(x,1)|dx\ll
\int_0^{\infty}
\frac{((3/2-x^2)^2+6x^2\cos^2\gamma )^{1/2}}{((1-x^2)^2+4x^2\sin^2\gamma )^{3/2}}xdx\ll\\
\ll
\int_0^{\infty}\min\left(T^3,\frac{1}{|1-x|^3}\right)dx\ll T^2.
\end{multline}

Now let us consider the first summand on the right-hand side of \eqref{application of approx fun.eq3}.
It follows from \eqref{Phi(x,1)def} that
\begin{equation}\label{Phi(x,1) integral}
\int_0^{\infty}\Phi(x,1)dx=
\frac{\sinh^2\beta}{2\pi c^2}\int_0^{\infty}
\frac{1-x^2/(4c^2)}{(1+x^2/(4c^2))^{2}}dx.
\end{equation}
In order to evaluate the integral above we use  \cite[page 311, Eq. (30)]{BE}, namely
\begin{equation}\label{Phi(x,1) integral2}
\int_0^{\infty}\frac{x^{s-1}}{(1+\alpha x^h)^{\nu}}dx=
h^{-1}\alpha^{-s/h}B(s/h,\nu-s/h),
\end{equation}
where $B(x,y)$ is the beta function, $|\arg\alpha|<\pi$, $h>0$, $0<\Re{s}<h\Re{\nu}.$
Applying \eqref{Phi(x,1) integral2} twice we show that
\begin{equation}\label{Phi(x,1) integral3}
\int_0^{\infty}\Phi(x,1)dx=0.
\end{equation}
Substituting \eqref{|Phi'(x,1)| integral} and  \eqref{Phi(x,1) integral3} in \eqref{application of approx fun.eq3} gives
\begin{equation}\label{application of approx fun.eq5}
\sum_{n=3}^{\infty}\Phi(n,1)S_V(n^2-4)\ll
V^{1/2}T^2\log^2V.
\end{equation}
Substituting \eqref{application of approx fun.eq5} to \eqref{application of approx fun.eq2} yields
\begin{equation}\label{application of approx fun.eq6}
\sum_{n=3}^{\infty}\mathscr{L}_{n^2-4}(1)\Phi(n,1)\ll
V^{1/2}T^2\log^2V+
V^{-1/2}\left(X^{\theta}T^2+X^{1/2+\theta}T\right).
\end{equation}

Now \eqref{estimate on sum L(1)} follows from \eqref{application of approx fun.eq6} by letting
\begin{equation*}
V=X^{\theta}(1+X^{1/2}/T).
\end{equation*}
\end{proof}
\section{Proof of Theorems \ref{mainthm} and \ref{thm:spec.exp.sum new bound2}}
Theorem \ref{mainthm}  follows directly from Lemma \ref{lemma estimates on the first error} and Lemma \ref{lemma estimate on the main error}.

Applying Theorem \ref{mainthm} to estimate the term $\mathcal{A}_1$ in \eqref{PrG1 =A+B+O} we obtain
\begin{multline}\label{spec.sumestimate1}
\sum_{j}\hat{\varphi}(t_j)\ll \mathcal{B}_1+
\max(X^{1/4+\theta/2}T^{3/2},X^{\theta/2}T^2)\log^2(XT)+\\
+O\left(N^{-1/2}X^{1/4+\theta}T^{3/2}\left(1+\frac{T}{X^{1/2}}\right)\right)
+\\
O\left(T^2\log^2T+\frac{N^{1/2}\log^2N}{X^{1/2}}\right).
\end{multline}

The part $\mathcal{B}_1$, defined by \eqref{B1 def}, can be bounded, for example, using \eqref{phihat expression} and the estimate of Luo and Sarnak \cite{LuoSarnakPG}. This yields
\begin{equation}\label{LuoSarnak Bbound for B1}
\mathcal{B}_1\ll N^{-1/2}T^{3}\log^2T.
\end{equation}
Nevertheless, as mentioned in the introduction, it is sufficient to use an analogue of the original result  of Iwaniec \eqref{Iwaniec Bbound}, namely
\begin{equation}\label{Iwaniec Bbound for B1}
\mathcal{B}_1\ll N^{-1/2}T^{7/2}\log^2T.
\end{equation}
Accordingly, substituting \eqref{Iwaniec Bbound for B1} to \eqref{spec.sumestimate1}, we obtain
\begin{multline}\label{spec.sumestimate2}
\sum_{j}\hat{\varphi}(t_j)\ll
\max(X^{1/4+\theta/2}T^{3/2},X^{\theta/2}T^2)\log^2(XT)+\\
\frac{T^{7/2}}{N^{1/2}}\log^2T
+N^{-1/2}X^{1/4+\theta}T^{3/2}\left(1+\frac{T}{X^{1/2}}\right)
+\\
T^2\log^2T+\frac{N^{1/2}\log^2N}{X^{1/2}}.
\end{multline}
Hence the optimal choice of $N$ is
\begin{equation}\label{N optimal}
N=T^{7/2}X^{1/2}+X^{3/4+\theta}T^{3/2}\left(1+\frac{T}{X^{1/2}}\right).
\end{equation}
It should be noted that $N$ is supposed to be an integer. Another important observation is that the summand $X^{3/4+\theta}T^{5/2}X^{-1/2}$ is negligible in comparison with the other two summands in \eqref{N optimal}. Thus we can simplify the choice of $N$ as follows
\begin{equation}\label{N choosen}
N=[T^{7/2}X^{1/2}+X^{3/4+\theta}T^{3/2}],
\end{equation}
where for a real number $x$, $[x]$ denotes the integral part of $x$. Substituting \eqref{N choosen} to \eqref{spec.sumestimate2}, we obtain
\begin{multline}\label{spec.sumestimate3}
\sum_{j}\hat{\varphi}(t_j)\ll
\max(X^{1/4+\theta/2}T^{3/2},X^{\theta/2}T^2)\log^2(XT)+\\
T^2\log^2T+
\left(
T^{7/4}X^{-1/4}+X^{-1/8+\theta/2}T^{3/4}
\right)\log^2(XT)\ll\\
\ll\max(X^{1/4+\theta/2}T^{3/2},X^{\theta/2}T^2)\log^2(XT).
\end{multline}

Applying \eqref{phihat expression} and the trivial bound $$\sum_{t_j}X^{it_j}\exp(-t_j/T)\ll T^2,$$ we prove that
\begin{equation}\label{spec.sumestimate2}
\sum_{t_j}t_jX^{it_j}\exp(-t_j/T)\ll \max\left(
X^{1/4+\theta/2}T^{3/2},
X^{\theta/2}T^2
\right)\log^{2}(XT).
\end{equation}
Now \eqref{spec.exp.sum new bound3} follows from \eqref{spec.sumestimate2}. See \cite[p. 235-236]{LuoSarnakPG} for details.


\section*{Acknowledgements}
The authors are grateful to the referee for very careful reading and comments that improved the exposition.



\nocite{}


\begin{thebibliography}{}

\bibitem{BF}
O. Balkanova and D. Frolenkov, \emph{Bounds for a spectral exponential sum}, J. London Math. Soc., published online, doi: 10.1112/jlms.12174 .

\bibitem{BE}
\newblock
H. Beitman and A. Erdelyi, \emph{Tables of integral transforms}, Vol. 1, McGraw-Hill, New York, 1954.


\bibitem{B}
\newblock
V. A. Bykovskii, \emph{Density theorems and the mean value of arithmetic functions on short intervals}. (Russian) Zap. Nauchn. Sem. S.-Peterburg. Otdel. Mat. Inst. Steklov. (POMI) 212 (1994), Anal. Teor. Chisel
i Teor. Funktsii. 12, 56--70, 196; translation in J. Math. Sci. (New York) 83 (1997), no. 6, 720--730.



\bibitem{Cai}
\newblock
Y. Cai, \emph{Prime geodesic theorem}, J. Theor.  Nombres Bordeaux 14:1 (2002), 59--72.


\bibitem{CI}
\newblock
J. B. Conrey and H. Iwaniec, \emph{The cubic moment of central values of automorphic
$L$-functions}, Ann. of Math. (2) 151 (2000), 1175--1216.

\bibitem{DesIw}
\newblock
J. M. Deshouillers and H. Iwaniec, \emph{The non-vanishing of Rankin-Selberg zeta-functions at special points}, Selberg trace formula and related topics, Contemp.  Math., 53, Amer. Math. Soc., Providence, RI, (1986), 59--95.



\bibitem{GR}
\newblock
I. S. Gradshteyn and I. M. Ryzhik, \emph{ Table of Integrals, Series, and Products}. Edited by A. Jeffrey and D. Zwillinger. Academic Press, New York, 7th edition (2007).

\bibitem{HMF}
F.W.J.~Olver , D.W.~Lozier, R.F.~Boisvert and C.W.~Clarke, \emph{{NIST} {H}andbook of {M}athematical {F}unctions}, Cambridge University Press, Cambridge (2010).

\bibitem{Iw}
\newblock
 H. Iwaniec, \emph{Topics in Classical Automorphic Forms}, Graduate studies in mathematics (vol 17), American Mathematical Soc., 1997.

\bibitem{IwPG}
\newblock
 H. Iwaniec, \emph{Prime geodesic theorem}, J. Reine. Angew. Math. 349 (1984), 136--159.


\bibitem{K}
N.V. Kuznetsov, \emph{An arithmetical form of the Selberg trace formula and the distribution of norms of primitive hyperbolic classes of the modular group}, in Russian, preprint Khabarovsk (1978).




\bibitem{LuoSarnakPG}
\newblock
 W. Luo and P. Sarnak, \emph{Quantum ergodicity of eigenfunctions on $PSL_2(Z)/H^2$}, Pub. math. de l'I.H.E.S. 81 (1995), 207--237.

\bibitem{PetRisLaak}
\newblock
Y. N. Petridis and M. S. Risager, \emph{Local average in hyperbolic lattice point counting, with an appendix by Niko Laaksonen.}  Math. Z. 285 (2017), no. 3-4, 1319--1344.

\bibitem{SY}
K. Soundararajan and  M. P. Young, \emph{The prime geodesic theorem}, J. Reine Angew. Math. 676 (2013), 105--120.

\bibitem{Y}
\newblock
M.P. Young, \emph{Weyl-type hybrid subconvexity bounds for twisted $L$-functions and Heegner points on shrinking sets}, J. Eur. Math. Soc. 19 (2017), 1545--1576.


\bibitem{Z}
\newblock
D. Zagier, \emph{ Modular forms whose Fourier coefficients involve zeta-functions of quadratic fields}. Modular
functions of one variable, VI (Proc. Second Internat. Conf., Univ. Bonn, Bonn, 1976), pp. 105--169.
Lecture Notes in Math., Vol. 627, Springer, Berlin, 1977.



\end{thebibliography}
\end{document}